\documentclass[a4paper, 12pt]{article}
\usepackage{}

\usepackage{mathrsfs}
\usepackage{epsfig}
\usepackage{amsmath}
\usepackage{amssymb}
\usepackage{latexsym}
\usepackage{amsfonts}
\usepackage{amsthm}

\usepackage[numbers,sort&compress]{natbib}
\usepackage{graphicx}
\ifx\pdfoutput\undefined  \else
 \fi

\usepackage[centerlast]{caption2}

\usepackage{color}

\usepackage{txfonts}
\usepackage{amssymb}
\usepackage{mathrsfs}
\usepackage{amsfonts}

\newtheorem{theorem}{Theorem}[section]
\newtheorem{lemma}[theorem]{Lemma}

\newtheorem{conj}[theorem]{Conjecture}

\usepackage{latexsym,amssymb}
\usepackage{amsmath}

\newcommand{\F}{{\mathbb F}}

\begin{document}

\title{Erd\H{o}s-Burgess constant of the multiplicative semigroup of the quotient ring of $\mathbb{F}_q[x]$}

\author{
Jun Hao$^{a}$  \ \ \ \ \ \
Haoli Wang$^{b}$\thanks{Corresponding author's Email:
bjpeuwanghaoli@163.com}  \ \ \ \ \ \
Lizhen Zhang$^{a}$
\\
\\
$^{a}${\small Department of Mathematics, Tianjin Polytechnic University, Tianjin, 300387, P. R. China}\\
$^{b}$ {\small College of Computer and Information Engineering}\\
{\small Tianjin Normal University, Tianjin, 300387, P. R. China}\\
}

\date{}
\maketitle

\begin{abstract}  Let $\mathcal{S}$ be a  semigroup endowed with a binary associative operation $*$. An element $e$ of $\mathcal{S}$ is said to be idempotent if $e*e=e$.  The {\sl Erd\H{o}s-Burgess constant}
of the semigroup $\mathcal{S}$ is defined as the smallest $\ell\in \mathbb{N}\cup \{\infty\}$ such that
any sequence $T$ of terms from $S$ and of length $\ell$ contains a nonempty subsequence the product of whose terms, in some order, is idempotent.  Let $q$ be a prime power, and let $\F_q[x]$ be the ring of polynomials over the finite field $\F_q$. Let $R=\F_q[x]\diagup K$ be a quotient ring of $\F_q[x]$ modulo any ideal $K$. We gave a sharp lower bound of the Erd\H{o}s-Burgess constant of the multiplicative semigroup of the ring $R$, in particular, we
determined the Erd\H{o}s-Burgess constant in the case when $K$ is factored into either a power of some prime ideal or a product of some pairwise distinct prime ideals in $\F_q[x]$.
\end{abstract}


\section {Introduction}

Let $\mathcal{S}$ be a nonempty semigroup, endowed with a binary associative operation $*$ on  $\mathcal{S}$, and denote by $E(\mathcal{S})$ the set of idempotents of $\mathcal{S}$, where $x\in \mathcal{S}$ is said to be an idempotent if $x*x=x$.
P. Erd\H{o}s posed a question on idempotent to D.A. Burgess as follows.

{\sl ``If $\mathcal{S}$ is a finite nonempty semigroup of order $n$, does any $\mathcal{S}$-valued sequence $T$ of length $n$ contain a nonempty subsequence the product of whose terms, in some order, is an idempotent?''}

In 1969, Burgess \cite{Burgess69} answered this question in the case when $\mathcal{S}$ is commutative or contains only one idempotent.  This question was completely affirmed by
D.W.H. Gillam, T.E. Hall and N.H. Williams, who proved the following stronger result:

\noindent \textbf{Theorem A.} (\cite{Gillam72}) \ {\sl Let $\mathcal{S}$ be a finite nonempty semigroup. Any $\mathcal{S}$-valued sequence of length  $|\mathcal{S}|-|E(\mathcal{S})|+1$ contains one or more terms whose product (in the order induced from the sequence $T$) is an idempotent; In addition, the bound $|\mathcal{S}|-|E(\mathcal{S})|+1$ is optimal.}

G.Q. Wang \cite{wangStruucture} generalized the result in the context of arbitrary semigroups (including both finite and infinite semigroups).

\noindent \textbf{Theorem B.} (\cite{wangStruucture}, Theorem 1.1) \ {\sl Let $\mathcal{S}$ be a nonempty semigroup such that $|\mathcal{S}\setminus E(\mathcal{S})|$ is finite. Any sequence $T$ of terms from $\mathcal{S}$ of length $|T|\geq|\mathcal{S}\setminus E(\mathcal{S})|+1$ contains one or more terms whose product (in the order induced from the sequence $T$) is an idempotent. }

Moreover, Wang \cite{wangStruucture} characterized the structure of extremal sequences of length $|\mathcal{S}\setminus E(\mathcal{S})|$ and remarked that although the bound $|\mathcal{S}\setminus E(\mathcal{S})|+1$ is
optimal for general semigroups $\mathcal{S}$, the better bound can be obtained for specific classes of semigroups. Hence, Wang proposed two combinatorial additive constants associated with idempotents.

\noindent \textbf{Definition C.} (\cite{wangStruucture}, Definition 4.1) \ {\sl Let $\mathcal{S}$ be a nonempty semigroup and $T$  a sequence of terms from $\mathcal{S}$. We say that $T$ is an {\bf idempotent-product sequence} if its terms can be ordered so that their product is an idempotent
element of $\mathcal{S}$. We call $T$ (weakly) {\bf idempotent-product free} if $T$ contains no nonempty idempotent-product subsequence, and we call $T$ {\bf strongly idempotent-product free} if $T$ contains no nonempty subsequence the product whose terms, in the order induced from the sequence $T$, is an idempotent.
We define $\textsc{I}(\mathcal{S})$, which is called the {\bf Erd\H{o}s-Burgess constant} of the semigroup $\mathcal{S}$, to be the least $\ell\in\mathbb{N}\cup \{\infty\}$ such that every sequence $T$ of terms from $\mathcal{S}$ of length at least $\ell$ is not (weakly) idempotent-product free,
and we define $\textsc{SI}(\mathcal{S})$, which is called the {\bf strong Erd\H{o}s-Burgess constant} of the semigroup $\mathcal{S}$, to be the least $\ell\in\mathbb{N}\cup \{\infty\}$ such that every sequence $T$ of terms from $\mathcal{S}$ of length at least $\ell$ is not strongly idempotent-product free. Formally, one can also define $$\textsc{I}(\mathcal{S})={\rm sup}\ \{|T|+1: T \mbox{ takes all idempotent-product free sequences of terms from } \mathcal{S}\}$$ and $$\textsc{SI}(\mathcal{S})={\rm sup}\ \{|T|+1: T \mbox{ takes every strongly idempotent-product free sequences of terms from } \mathcal{S}\}.$$}

Very recently, Wang \cite{wangErdos-burgess} made a comprehensive study of the Erd\H{o}s-Burgess constant for the direct product of arbitrarily many of cyclic semigroups.
As pointed out in \cite{wangStruucture}, the Erd\H{o}s-Burgess constant reduces to be the famous Davenport constant in the case when the underlying semigroup happens to be a finite abelian group. So we need to introduce the definition of Davenport constant below.

Let $G$ be an additive finite abelian group. A sequence $T$ of
terms from $G$ is called a {\sl zero-sum sequence} if the sum of
all terms of $T$ equals to zero, the identity element of $G$. We call $T$ a {\sl zero-sum free} sequence if $T$ contains no nonempty  zero-sum subsequence.
The Davenport constant ${\rm D}(G)$ of
$G$ is defined to be the smallest positive integer $\ell$  such that,
every sequence $T$ of terms from $G$ and of length at least $\ell$ is not zero-sum free.

In 2008, Wang and Gao \cite{wanggao} extended the definition of the Davenport constant to commutative semigroups as follows.

\noindent \textbf{Definition D.} \ {\sl Let $\mathcal{S}$ be a finite commutative semigroup. Let $T$ be a sequence of terms from the semigroup $\mathcal{S}$. We call $T$  reducible if $T$ contains a proper subsequence $T'$ ($T'\neq T$) such that the sum of all terms of $T'$ equals the sum of all terms of $T$. Define the Davenport constant of the semigroup $\mathcal{S}$, denoted ${\rm D}(\mathcal{S})$, to be the smallest $\ell\in \mathbb{N}\cup\{\infty\}$ such that every sequence $T$ of length at least $\ell$ of terms from $\mathcal{S}$ is reducible.}

Several related additive results on Davenport constant for semigroups were obtained  (see \cite{wangDavenportII}, \cite{wangAddtiveirreducible},  \cite{gaowangII}, \cite{wang-zhang-qu}). For any commutative ring $R$, we denote $\mathcal{S}_R$ to be the multiplicative semigroup of the ring $R$ and ${\rm U}(\mathcal{S}_R)$ to be the group of units of the semigroup $\mathcal{S}_R$.
With respect to the Davenport constant for the multiplicative semigroup associated with polynomial rings $\F_q[x]$, Wang obtained the following result.

\noindent \textbf{Theorem E}. (\cite{wangDavenportII}) \ {\sl Let $q>2$ be a prime power, and let $\F_q[x]$ be the ring of polynomials over the finite field $\F_q$.
Let $R$ be a quotient ring of $\F_q[x]$ with $0\neq R\neq \F_q[x]$. Then ${\rm D}(\mathcal{S}_R)={\rm D}({\rm U}(\mathcal{S}_R)).$}

G.Q. Wang \cite{wangDavenportII} proposed to determine ${\rm D}(\mathcal{S}_R)-{\rm D}({\rm U}(\mathcal{S}_R))$ for the remaining case that $R$ is a quotient ring of $\mathbb{F}_2[x]$.

L.Z. Zhang, H.L. Wang and Y.K. Qu partially answered  Wang's question and obtained the following.

\noindent \textbf{Theorem F}. (\cite{wang-zhang-qu}) \ {\sl  Let $\F_2[x]$ be the ring of polynomials over the finite field $\F_2$, and let $R={\F_2[x]}\diagup{(f)}$ be a quotient ring of $\F_2[x]$, where $f\in\F_2[x]$ and $0\neq R\neq \F_2[x]$.
Then  $${\rm D}({\rm U}(\mathcal{S}_R))\leq {\rm D}(\mathcal{S}_R)\leq {\rm D}({\rm U}(\mathcal{S}_R))+\delta_f,$$
where
\begin{displaymath}
\delta_f=\left\{ \begin{array}{ll}
0 & \textrm{if $\gcd(x*(x+1_{\mathbb{F}_2}),\ f)=1_{\F_{2}}$;}\\
1 & \textrm{if $\gcd(x*(x+1_{\mathbb{F}_2}),\ f)\in \{x, \ x+1_{\mathbb{F}_2}\}$;}\\
2 & \textrm{if $\gcd(x*(x+1_{\mathbb{F}_2}),f)=x*(x+1_{\mathbb{F}_2})$.}\\
\end{array} \right.
\end{displaymath}}

Motivated by the above additive research on semigroups, in this manuscript we make a study of the Erd\H{o}s-Burgess constant on the multiplicative semigroups of the quotient rings of the polynomial rings $\F_q[x]$ and obtain the following result.

\begin{theorem}\label{Theorem main}
\  Let $q$ be a prime power, and let $\F_q[x]$ be the ring of polynomials over the finite field $\F_q$. Let $R=\F_q[x]\diagup K$ be a quotient ring of $\F_q[x]$ modulo any ideal $K$. Then $${\rm I}(\mathcal{S}_R)\geq {\rm D}({\rm U}(\mathcal{S}_R))+\Omega(K)-\omega(K),$$ where $\Omega(K)$ is the number of the prime ideals (repetitions are counted) and $\omega(K)$ the number of distinct prime ideals
in the factorization when $K$
is factored into a product of prime ideals. Moreover, the equality holds for the case when
$K$ is factored into either a power of some prime ideal or a product of some pairwise distinct prime ideals in $\F_q[x]$.
\end{theorem}

\section{Notation}

Let $\mathcal{S}$ be a finite commutative semigroup.
The operation on $\mathcal{S}$ will be denoted by $+$ instead of $*$.
The identity element of $\mathcal{S}$, denoted $0_{\mathcal{S}}$ (if exists), is the unique element $e$ of
$\mathcal{S}$ such that $e+a=a$ for every $a\in \mathcal{S}$. If $\mathcal{S}$ has an identity element $0_{\mathcal{S}}$, let
$${\rm U}(\mathcal{S})=\{a\in \mathcal{S}: a+a'=0_{\mathcal{S}} \mbox{ for some }a'\in \mathcal{S}\}$$ be the group of units
of $\mathcal{S}$.
The sequence $T$ of terms from the semigroups $\mathcal{S}$ is denoted by $$T=a_1a_2\cdot\ldots\cdot a_{\ell}=\coprod\limits_{a\in \mathcal{S}} a^{\ {\rm v}_a(T)},$$ where ${\rm v}_a(T)$ denotes the multiplicity of the element $a$ occurring in the sequence $T$. By $\cdot$ we denote the operation to join sequences.
Let $T_1,T_2$ be two sequences of terms from the semigroups $\mathcal{S}$. We call $T_2$
a subsequence of $T_1$ if $${\rm v}_a(T_2)\leq {\rm v}_a(T_1)$$ for every element $a\in \mathcal{S}$, denoted by $$T_2\mid T_1.$$ In particular, if $T_2\neq T_1$, we call $T_2$ a {\sl proper} subsequence of $T_1$, and write $$T_3=T_1  T_2^{-1}$$ to mean the unique subsequence of $T_1$ with $T_2\cdot T_3=T_1$.  Let $$\sigma(T)=a_1+a_2+\cdots+a_{\ell}$$ be the sum of all terms in the sequence $T$.

Let $q$ be a prime power, and let $\F_q[x]$ be the ring of polynomials over the finite field $\F_q$. Let $R=\F_q[x]\diagup K$ be the quotient ring of $\F_q[x]$ modulo the ideal $K$, and let $\mathcal{S}_R$ be the multiplicative semigroup of the ring $R$.
Take an arbitrary element $a\in \mathcal{S}_{R}$.
Let $\theta_a\in \mathbb{F}_q[x]$ be the unique polynomial corresponding to the element $a$ with the least degree, thus,
$\overline{\theta_a}=\theta_a+K$ is the corresponding form of $a$ in the quotient ring $R$.

\noindent $\bullet$  In what follows, since we deal with only the multiplicative semigroup $\mathcal{S}_{R}$ which happens to be commutative,  we shall use the terminology {\sl idempotent-sum} and {\sl idempotent-sum free} in place of {\sl idempotent-product} and {\sl idempotent-product free}, respectively.

\section{Proof of Theorem \ref{Theorem main}}

\begin{lemma} \label{Lemma idempotent form} Let $q$ be a prime power, and let $\F_q[x]$ be the ring of polynomials over the finite field $\F_q$.
Let $f$ be a polynomial in $\F_q[x]$ and let
$f=p_1^{n_1}p_2^{n_2}\cdots p_r^{n_r}$,
where $r\geq 1$, $n_1,n_2,\ldots,n_r\geq 1$, and
$p_1, p_2, \ldots,p_r$ are pairwise non-associate  irreducible polynomials in $\mathbb{F}_q[x]$. Let $R=\mathbb{F}_q[x]\diagup (f)$ be the quotient ring of $\mathbb{F}_q[x]$ modulo the ideal $(f)$.
Let $a$ be an element in the semigroup of $\mathcal{S}_R$.
Then $a$ is idempotent if and only if $\theta_a\equiv 0_{\mathbb{F}_q}\pmod {p_i^{n_i}}$ or $\theta_a\equiv 1_{\mathbb{F}_q}\pmod {p_i^{n_i}}$ for every $i\in [1,r]$.
\end{lemma}

\begin{proof} \ Suppose that $a$ is idempotent. Then $\theta_a\theta_a\equiv \theta_a\pmod f$, which implies that $\theta_a(\theta_a-1_{\mathbb{F}_q})\equiv 0_{\mathbb{F}_q} \pmod {p_i^{n_i}}$ for all $i \in [1,r]$. Since $\gcd(\theta_a, \theta_a-1_{\mathbb{F}_q})=1_{\mathbb{F}_q}$,  it follows that for every $i\in [1,r]$, $p_i^{n_i}$ divides $\theta_a$ or $p_i^{n_i}$ divides $\theta_a-1_{\mathbb{F}_q}$, that is, $\theta_a\equiv 0_{\mathbb{F}_q}\pmod {p_i^{n_i}}$ or $\theta_a\equiv 1_{\mathbb{F}_q}\pmod {p_i^{n_i}}$. Then the necessity holds. The sufficiency holds similarly.
\end{proof}

We remark that in Theorem \ref{Theorem main}, if $K=\F_q[x]$, then $R$ is a trivial zero ring and ${\rm I}(\mathcal{S}_R)={\rm D}(\mathcal{S}_R)=1$ and $\Omega(K)=\omega(K)=0$,
and if $K$ is the zero ideal then $R=\F_q[x]$ and ${\rm I}(\mathcal{S}_R)$ is infinite since any sequence $T$ of any length such that $\theta_a$ is a nonconstant polynomial for all terms $a$ of $T$ is an idempotent-sum free sequence, and thus, the conclusion holds trivially for both cases.    Hence, we shall only consider the case that $K$ is nonzero proper ideal of $\F_q[x]$ in what follows.

\noindent {\sl Proof of Theorem  \ref{Theorem main}.} \ Note that $\F_q[x]$ is a principal ideal domain. Say
\begin{equation}\label{equation K=(f)}
K=(f)
\end{equation} is the principal ideal generated by a polynomial $f\in \F_q[x]$, where
\begin{equation}\label{equation factorization of f(x)}
f=p_1^{n_{1}}p_2^{n_{2}}\cdots p_r^{n_r},
\end{equation}
 where $p_1,p_2, \ldots,p_r$ are pairwise non-associate irreducible polynomials of $\mathbb{F}_q[x]$ and $n_i\geq 1$ for all $i\in [1,r]$, equivalently, $$K=P_1^{n_1}P_2^{n_2}\cdots P_r^{n_r}$$ is the factorization of the ideal $K$ into the product of the powers of distinct prime ideals $P_1=(p_1),P_2=(p_2),\ldots,P_r=(p_r).$  Observe that
 \begin{equation}\label{equation bigomega(K)}
 \Omega(K)=\sum\limits_{i=1}^r n_i
 \end{equation}
 and \begin{equation}\label{equation smallomega(K)}
 \omega(K)=r.
 \end{equation}

 Take a zero-sum free sequence $V$ of terms from the group ${\rm U}(\mathcal{S}_R)$ of length ${\rm D}({\rm U}(\mathcal{S}_R))-1$. Take $b_i\in \mathcal{S}_R$ such that
$\theta_{b_i}=p_i$ for each $i\in [1,r]$.
Now we show that the sequence $V\cdot \coprod\limits_{i=1}^r b_i^{n_i-1}$ is an idempotent-sum free sequence in $\mathcal{S}_R$. Suppose to the contrary that $V\cdot \coprod\limits_{i=1}^r b_i^{n_i-1}$ contains a {\bf nonempty} subsequence $W$, say $W=V'\cdot \coprod\limits_{i=1}^r b_i^{\beta_i} $,  such that $\sigma(W)$ is idempotent, where $V'$ is a subsequence of $V$ and $$\beta_i\in [0,n_i-1]\mbox{ for all } i\in [1,r].$$ It follows that
\begin{equation}\label{equation theta sigma(W)}
\theta_{ \sigma(W)}=\theta_{ \sigma(V')}\theta_{ \sigma(\coprod\limits_{i=1}^r b_i^{\beta_i})}=\theta_{\sigma(V')}p_1^{\beta_1}\cdots p_r^{\beta_r}.
\end{equation}
If $\sum\limits_{i=1}^r\beta_i=0$, then $W=V'$ is a {\sl nonempty} subsequence of $V$. Since $V$ is zero-sum free in the group of ${\rm U}(\mathcal{S}_R)$, we derive that $\sigma(W)$ is a nonidentity element of the group ${\rm U}(\mathcal{S}_R)$, and thus, $\sigma(W)$ is not idempotent, a contradiction. Otherwise, $\beta_j>0$ for some $j\in [1,r]$, say
\begin{equation}\label{equation beta1 in [1,n1-1]}
\beta_1\in [1,n_1-1].
\end{equation}
 Since $\gcd(\theta_{\sigma(V')},p_1)=1_{\mathbb{F}_q}$, it follows from \eqref{equation theta sigma(W)} that $\gcd(\theta_{\sigma(W)},p_1^{n_1})=p_1^{\beta_1}$. Combined with \eqref{equation beta1 in [1,n1-1]}, we have that $\theta_{\sigma(W)}\not\equiv 0_{\mathbb{F}_q}\pmod {p_1^{n_1}}$ and
$\theta_{\sigma(W)}\not\equiv 1_{\mathbb{F}_q}\pmod {p_1^{n_1}}$. By
Lemma \ref{Lemma idempotent form}, we conclude that $\sigma(W)$ is not idempotent, a contradiction.
This proves that the sequence $V\cdot \coprod\limits_{i=1}^r b_i^{n_i-1}$ is idempotent-sum free in $\mathcal{S}_R$. Combined with \eqref{equation bigomega(K)} and \eqref{equation smallomega(K)}, we have that
\begin{equation}\label{equation I(S)geq in case prime power}
{\rm I}(\mathcal{S}_R)\geq |V\cdot \coprod\limits_{i=1}^r b_i^{n_i-1}|+1=(|V|+1)+\sum\limits_{i=1}^r (n_i-1)=
{\rm D}({\rm U}(\mathcal{S}_R))+\Omega(K)-\omega(K).
\end{equation}

Now we assume that $K$ is factored into either a power of some prime ideal or a product of some pairwise distinct prime ideals in $\F_q[x]$, i.e., either $r=1$ or $n_1=\cdots =n_r=1$ in \eqref{equation factorization of f(x)}. It remains to show the equality ${\rm I}(\mathcal{S}_R)=
{\rm D}({\rm U}(\mathcal{S}_R))+\Omega(K)-\omega(K)$ holds. We distinguish two cases.

\noindent \textbf{Case 1.} \ $r=1$ in \eqref{equation factorization of f(x)}, i.e., $f=p_1^{n_1}$.

Take an arbitrary sequence $T$ of length $|T|={\rm D}({\rm U}(\mathcal{S}_R))+n_1-1={\rm D}({\rm U}(\mathcal{S}_R))+\Omega(K)-\omega(K)$.
 Let $T_1=\coprod\limits_{\stackrel{a\mid T}{\theta_a\equiv 0 \pmod {p_1}}} a$ and $T_2=T T_1^{-1}$. Note that all terms of $T_2$ are from ${\rm U}(\mathcal{S}_R)$.
 By the Pigeonhole Principle, we see that either $|T_1|\geq n_1$ or $|T_2|\geq {\rm D}({\rm U}(\mathcal{S}_R))$. It follows that either $\theta_{\sigma(T_1)}\equiv 0_{\mathbb{F}_q}\pmod {p_1^{n_1}}$, or $T_2$ contains a nonempty subsequence $T_2'$ such that $\sigma(T_2')$ is the identity element of the group ${\rm U}(\mathcal{S}_R)$. By Lemma \ref{Lemma idempotent form}, the sequence $T$ is not idempotent-sum free, which implies that ${\rm I}(\mathcal{S}_R)\leq {\rm D}({\rm U}(\mathcal{S}_R))+\Omega(K)-\omega(K)$. Combined with \eqref{equation I(S)geq in case prime power}, we have that $${\rm I}(\mathcal{S}_R)={\rm D}({\rm U}(\mathcal{S}_R))+\Omega(K)-\omega(K).$$

\noindent \textbf{Case 2.} \ $n_1=\cdots =n_r=1$ in \eqref{equation factorization of f(x)}, i.e., $f=p_1p_2\cdots p_r$.

Then
\begin{equation}\label{equation bigomega=smallomega}
\Omega(K)=\omega(K)=r.
\end{equation}
Take an arbitrary sequence $T$ of length $|T|={\rm D}({\rm U}(\mathcal{S}_R))$.
For any term $a$ of $T$, let $\widetilde{a}\in \mathcal{S}_R$ be such that for each $i\in [1,r]$,
\begin{equation}\label{equaiton tilde a}
\theta_{\widetilde{a}}\equiv\left\{ \begin{array}{ll}
1_{\mathbb{F}_q} \pmod {p_i} & \textrm{if $\theta_{a}\equiv 0_{\mathbb{F}_q}\pmod {p_i}$;}\\
\theta_{a} \pmod {p_i} & \textrm{otherwise.}\\
\end{array} \right.
\end{equation}
Note that $$\widetilde{a}\in {\rm U}(\mathcal{S}_R).$$ Let $\widetilde{T}=\coprod\limits_{a\mid T}\tilde{a}$. Then $\widetilde{T}$ is a sequence of terms from the group ${\rm U}(\mathcal{S}_R)$ with length $|\widetilde{T}|=|T|= {\rm D}({\rm U}(\mathcal{S}_R))$. It follows that there exists a nonempty subsequence $W$ of $T$ such that $\sigma(\coprod\limits_{a\mid W}\tilde{a})$ is the identity element of the group ${\rm U}(\mathcal{S}_R)$, i.e., $\theta_{\sigma(\coprod\limits_{a\mid W}\tilde{a})}\equiv 1_{\mathbb{F}_q}\pmod {p_i}$ for each $i\in [1,r]$. By \eqref{equaiton tilde a}, we derive that $\theta_{\sigma(W)}\equiv 0_{\mathbb{F}_q}\pmod {p_i}$ or $\theta_{\sigma(W)}\equiv 1_{\mathbb{F}_q}\pmod {p_i}$ for each $i\in [1,r]$.
By Lemma \ref{Lemma idempotent form}, we conclude that $\sigma(W)$ is idempotent. Combined with \eqref{equation bigomega=smallomega}, we have that ${\rm I}( \mathcal{S}_R)\leq {\rm D}({\rm U}(\mathcal{S}_R))={\rm D}({\rm U}(\mathcal{S}_R))+\Omega(K)-\omega(K)$. It follows from \eqref{equation I(S)geq in case prime power} that ${\rm I}(\mathcal{S}_R)={\rm D}({\rm U}(\mathcal{S}_R))+\Omega(K)-\omega(K)$, completing the proof. \qed

We close this paper with the following conjecture.

\begin{conj}
\  Let $q>2$ be a prime power, and let $\F_q[x]$ be the ring of polynomials over the finite field $\F_q$. Let $R=\F_q[x]\diagup K$ be a quotient ring of $\F_q[x]$ modulo any nonzero proper ideal $K$. Then
${\rm I}(\mathcal{S}_R)={\rm D}({\rm U}(\mathcal{S}_R))+\Omega(K)-\omega(K).$
\end{conj}

\bigskip

\noindent {\bf Acknowledgements}

\noindent
This work is supported by NSFC (grant no. 11501561, 61303023).


\begin{thebibliography}{99}

\bibitem{Burgess69} D.A. Burgess, \emph{A problem on semi-groups}, Studia Sci. Math. Hungar., \textbf{4} (1969) 9--11.

\bibitem{Gillam72} D.W.H. Gillam, T.E. Hall and N.H. Williams, \emph{On finite semigroups and idempotents}, Bull. Lond. Math. Soc., \textbf{4} (1972) 143--144.



\bibitem{rog1}
K. Rogers, \emph{A Combinatorial problem in Abelian groups,} Proc.
Cambridge Phil. Soc., \textbf{59} (1963) 559--562.

\bibitem{wangDavenportII}  G.Q. Wang, \emph{Davenport constant for semigroups II,}  J. Number Theory, \textbf{153} (2015) 124--134.

\bibitem{wangAddtiveirreducible}  G.Q. Wang, \emph{Additively irreducible sequences in commutative semigroups,}  J. Combin. Theory Ser. A, \textbf{152} (2017)  380--397.

\bibitem{wangStruucture}  G.Q. Wang, \emph{Structure of the largest idempotent-free sequences in finite semigroups,}  arXiv:1405.6278.

\bibitem{wangErdos-burgess}  G.Q. Wang, \emph{Erd\H{o}s-Burgess constant of the direct product of cyclic semigroups,} arXiv:1802.08791.

\bibitem{wanggao} G.Q. Wang and W.D. Gao,
\emph{Davenport constant for semigroups,} Semigroup Forum,
\textbf{76} (2008) 234--238.

\bibitem{gaowangII} G.Q. Wang and W.D. Gao,
Davenport constant of the multiplicative semigroup of the ring $\mathbb{Z}_{n_1}\oplus\cdots\oplus \mathbb{Z}_{n_r}$, arXiv:1603.06030.

\bibitem{wang-zhang-qu} L.Z. Zhang, H.L. Wang and Y.K Qu,  \emph{A problem of Wang on Davenport constant for the multiplicative semigroup of the quotient ring of $\F_2[x]$},  Colloq. Math., \textbf{148} (2017) 123--130.


\end{thebibliography}
\end{document}